\documentclass{amsart}
\usepackage[english]{babel}
\usepackage[T1]{fontenc}
\usepackage[utf8]{inputenc}
\usepackage{amsmath}
\usepackage{amsfonts}
\usepackage{array}

\usepackage{amssymb}
\usepackage{mathrsfs}
\usepackage{amsthm}
\usepackage{xfrac}
\usepackage{lmodern}
\usepackage{fix-cm}
\usepackage{listings}
\usepackage{fancyvrb}
\usepackage{enumerate}   
\usepackage{xcolor}
\usepackage{pgf,tikz}
\usepackage{mathrsfs}
\usetikzlibrary{arrows}
\usepackage{float}
\usepackage{subcaption}

\usepackage{hyperref}{}
\usetikzlibrary{patterns}

\def\NZQ{\Bbb}               
\def\NN{{\NZQ N}}
\def\QQ{{\NZQ Q}}
\def\ZZ{{\NZQ Z}}

\def\PP{{\NZQ P}}

\def\frk{\frak}               

\def\Phi{{\frk n}}
\def\Phi{{\frk N}}
\def\MC{{\mathcal C}}

\def\MP{{\mathcal P}}
\def\PP{{\mathcal P}}

\def\MQ{{\mathcal Q}}

\def\ME{{\mathcal E}}

\def\MM{{\mathcal M}}
\def\MO{{\mathcal O}}
\def\MH{{\mathcal H}}

\def\opn#1#2{\def#1{\operatorname{#2}}} 
\opn\chara{char} \opn\length{\ell} \opn\pd{pd} \opn\rk{rk}
\opn\projdim{proj\,dim} \opn\injdim{inj\,dim} \opn\rank{rank}
\opn\depth{depth} \opn\grade{grade} \opn\height{height}
\opn\embdim{emb\,dim} \opn\codim{codim}

\opn\Tr{Tr} \opn\bigrank{big\,rank}
\opn\superheight{superheight}\opn\lcm{lcm}
\opn\trdeg{tr\,deg}
\opn\reg{reg} \opn\lreg{lreg} \opn\ini{in} \opn\lpd{lpd}
\opn\size{size}\opn\bigsize{bigsize}
\opn\cosize{cosize}\opn\bigcosize{bigcosize}
\opn\sdepth{sdepth}\opn\sreg{sreg}
\opn\link{link}\opn\fdepth{fdepth}
\opn\div{div} \opn\Div{Div} \opn\cl{cl} \opn\Cl{Cl}
\opn\Spec{Spec} \opn\Supp{Supp} \opn\supp{supp} \opn\Sing{Sing}
\opn\Ass{Ass} \opn\Min{Min}\opn\Mon{Mon} \opn\dstab{dstab} \opn\astab{astab}
\opn\Ann{Ann} \opn\Rad{Rad} \opn\Soc{Soc}
\opn\Im{Im} \opn\Ker{Ker} \opn\Coker{Coker} \opn\Am{Am}
\opn\Hom{Hom} \opn\Tor{Tor} \opn\Ext{Ext} \opn\End{End}
\opn\Aut{Aut} \opn\id{id}

\opn\nat{nat}
\opn\pff{pf}
\opn\Pf{Pf} \opn\GL{GL} \opn\SL{SL} \opn\mod{mod} \opn\ord{ord}
\opn\Gin{Gin} \opn\Hilb{Hilb}\opn\sort{sort}
\opn\aff{aff} \opn\con{conv} \opn\relint{relint} \opn\st{st}
\opn\lk{lk} \opn\cn{cn} \opn\core{core} \opn\vol{vol}
\opn\link{link} \opn\star{star}\opn\lex{lex}
\opn\grevlex{grevlex}
\opn\cdeg{cdeg}
\opn\T{T}
\opn\gr{gr}

\def\pot#1#2{#1[\kern-0.28ex[#2]\kern-0.28ex]}

\opn\dirlim{\underrightarrow{\lim}}
\opn\inivlim{\underleftarrow{\lim}}
\def\Implies{\ifmmode\Longrightarrow \else
        \unskip${}\Longrightarrow{}$\ignorespaces\fi}
\def\implies{\ifmmode\Rightarrow \else
        \unskip${}\Rightarrow{}$\ignorespaces\fi}
\def\iff{\ifmmode\Longleftrightarrow \else
        \unskip${}\Longleftrightarrow{}$\ignorespaces\fi}

\let\:=\colon

\theoremstyle{plain}
\newtheorem{theorem}{Theorem}[section]
\newtheorem{lemma}[theorem]{Lemma}
\newtheorem{proposition}[theorem]{Proposition}
\newtheorem{corollary}[theorem]{Corollary}

\newtheorem*{theorem-q}{Theorem}
\newtheorem*{corollary-q}{Corollary}
\newtheorem*{question-q}{Question}
\newtheorem*{questions-q}{Questions}
\theoremstyle{definition}
\newtheorem{definition}[theorem]{Definition}

\theoremstyle{remark}

\newcommand{\resi}{0.1}
\newcommand{\resiII}{0.11}
\newcommand{\one}{ 
 \draw (1,0)--(1,2);
     \draw (0,1)--(2,1);
    \draw (0,0)--(2,0);
    \draw (2,0)--(2,2);  
  
  \draw (0,0)--(0,1);
    \draw (1,2)--(2,2);
    \draw[dashed] (0,1)--(0,2);
    \draw[dashed] (0,2)--(1,2);
    
    \node at (1,-0.2){};
    \node at (1,2.2){}; }
   	
    \newcommand{\two}{     
\draw (1,0)--(1,2);
     \draw (0,1)--(2,1);
    \draw (0,0)--(0,2);  
    \draw (0,2)--(2,2);  
  \draw (0,0)--(1,0);
    \draw (2,1)--(2,2);
    \draw[dashed] (1,0)--(2,0);
    \draw[dashed] (2,0)--(2,1);   
      \node at (1,-0.2){};
    \node at (1,2.2){};
    }
    
\newcommand{\three}{  \draw (1,0)--(1,2);
     \draw (0,1)--(2,1);
    \draw (0,0)--(0,2); 
    \draw (0,0)--(2,0);
  \draw (0,2)--(1,2);
    \draw (2,1)--(2,0);
    \draw[dashed] (1,2)--(2,2);
    \draw[dashed] (2,1)--(2,2);
      \node at (1,-0.2){};
    \node at (1,2.2){};
    }
    
    \newcommand{\four}{
         \draw (1,0)--(1,2);
     \draw (0,1)--(2,1);
    \draw (2,0)--(2,2);  
\draw (0,2)--(2,2);
  \draw (1,0)--(2,0);
    \draw (0,1)--(0,2);
    \draw[dashed] (0,0)--(1,0);
    \draw[dashed] (0,1)--(0,0);
    
      \node at (1,-0.2){};
    \node at (1,2.2){};
    }

    \newcommand{\oneI}{ \resizebox{\resi\textwidth}{!}{
\begin{tikzpicture}
  \two
 \filldraw (0,1) circle (2pt);
 \node at (0,1.2) [anchor=west]{$v$};
\end{tikzpicture} }}
    
 \newcommand{\oneII}{ \resizebox{\resi\textwidth}{!}{
\begin{tikzpicture}
	\four
 \filldraw (1,1) circle (2pt);
 \node at (1,1.2) [anchor=west]{$v$};
\end{tikzpicture} }}
    \newcommand{\twoI}{ \resizebox{\resi\textwidth}{!}{
\begin{tikzpicture}
\four
 \filldraw (2,1) circle (2pt);
 \node at (2,1.2) [anchor=east]{$v$};
\end{tikzpicture} }}
 \newcommand{\twoII}{ \resizebox{\resi\textwidth}{!}{
\begin{tikzpicture}
\two
 \filldraw (1,1) circle (2pt);
 \node at (1,1.2) [anchor=west]{$v$};
\end{tikzpicture} }}
    \newcommand{\threeI}{ \resizebox{\resi\textwidth}{!}{
\begin{tikzpicture}
\one
 \filldraw (2,1) circle (2pt);
 \node at (2,1.2) [anchor=east]{$v$};
\end{tikzpicture} }}

 \newcommand{\threeII}{\resizebox{\resi\textwidth}{!}{
\begin{tikzpicture}
\three
 \filldraw (1,1) circle (2pt);
 \node at (1,1.2) [anchor=west]{$v$};
\end{tikzpicture} }}
    \newcommand{\fourI}{ \resizebox{\resi\textwidth}{!}{
\begin{tikzpicture}
	\three
 \filldraw (0,1) circle (2pt);
 \node at (0,1.2) [anchor=west]{$v$};
\end{tikzpicture} }}
 \newcommand{\fourII}{\resizebox{\resi\textwidth}{!}{
\begin{tikzpicture}
\one
 \filldraw (1,1) circle (2pt);
 \node at (1,1.2) [anchor=west]{$v$};
\end{tikzpicture} }}
    \newcommand{\fiveI}{ \resizebox{\resi\textwidth}{!}{
\begin{tikzpicture}
  \one
 \filldraw (1,0) circle (2pt);
 \node at (1,0.2) [anchor=west]{$v$};
\end{tikzpicture} }}
 \newcommand{\fiveII}{ \resizebox{\resi\textwidth}{!}{
\begin{tikzpicture}
	\four
 \filldraw (1,1) circle (2pt);
 \node at (1,1.2) [anchor=west]{$v$};
\end{tikzpicture} }}
 \newcommand{\sixI}{ \resizebox{\resi\textwidth}{!}{
\begin{tikzpicture}
\three
 \filldraw (1,0) circle (2pt);
 \node at (1,0.2) [anchor=west]{$v$};
\end{tikzpicture} }}
 \newcommand{\sixII}{ \resizebox{\resi\textwidth}{!}{
\begin{tikzpicture}
\two
 \filldraw (1,1) circle (2pt);
 \node at (1,1.2) [anchor=west]{$v$};
\end{tikzpicture} }}
    \newcommand{\sevenI}{ \resizebox{\resi\textwidth}{!}{
\begin{tikzpicture}
\two
 \filldraw (1,2) circle (2pt);
 \node at (1,1.8) [anchor=west]{$v$};
\end{tikzpicture} }}
 \newcommand{\sevenII}{\resizebox{\resi\textwidth}{!}{
\begin{tikzpicture}
\three
 \filldraw (1,1) circle (2pt);
 \node at (1,1.2) [anchor=west]{$v$};
\end{tikzpicture} }}
    \newcommand{\eightI}{ \resizebox{\resi\textwidth}{!}{
\begin{tikzpicture}
\four
 \filldraw (1,2) circle (2pt);
 \node at (1,1.8) [anchor=west]{$v$};
\end{tikzpicture} }}
 \newcommand{\eightII}{\resizebox{\resi\textwidth}{!}{
\begin{tikzpicture}
\one
 \filldraw (1,1) circle (2pt);
 \node at (1,1.2) [anchor=west]{$v$};
\end{tikzpicture} }}

    \newcommand{\parI}{ \resizebox{\resi\textwidth}{!}{
\begin{tikzpicture}
 \node at (0,0) {{\scriptsize(I)}};
  \node at (0,-0.38){};
\end{tikzpicture} }}
 \newcommand{\parII}{\resizebox{\resiII\textwidth}{!}{
\begin{tikzpicture}
 \node at (0,0) {{\scriptsize(II)}};
 \node at (0,-0.38){};
\end{tikzpicture} }}

\renewcommand{\ME}{\mathcal{E}}

\renewcommand{\QQ}{\mathcal{Q}}
\renewcommand{\phi}{\varphi}

\title{Primality of polyomino ideals by quadratic Gr\"obner basis}

\author[Carla Mascia]{Carla Mascia}

\address[Carla Mascia]{Department of Mathematics\\
University of Trento\\
via Sommarive, 14\\
38123 Povo (Trento), Italy}
\email{carla.mascia@unitn.it}

\author[Giancarlo Rinaldo]{Giancarlo Rinaldo}

\address[Giancarlo Rinaldo]{Department of Mathematics\\
University of Trento\\
via Sommarive, 14\\
38123 Povo (Trento), Italy}
\email{giancarlo.rinaldo@unitn.it}

\author[Francesco Romeo]{Francesco Romeo}

\address[Francesco Romeo]{Department of Mathematics\\
University of Trento\\
via Sommarive, 14\\
38123 Povo (Trento), Italy}
\email{francesco.romeo-3@unitn.it}

\begin{document}

\begin{abstract}
In this work, we provide a necessary and sufficient condition on a polyomino ideal for having the set of inner $2$-minors as graded reverse lexicographic Gr\"obner basis, due to combinatorial properties of the polyomino itself. Moreover, we prove that when the latter holds the ideal  coincides with the lattice ideal associated to the polyomino, that is the ideal is prime. As an application, we describe two new infinite families of prime polyominoes.
\end{abstract}

\maketitle

\section*{Introduction}
The ideals generated by a subset of $t$-minors of an $m \times n$ matrix of indeterminates are an intensively-studied class of binomial ideals, due to their applications in algebraic statistics. Among these ideals, one finds the determinantal ideals, see, for instance, \cite{BV} and its references to original articles, the ladder ideals introduced by Conca in \cite{Co}, and the ideals of adjacent minors introduced by Ho\c{s}ten and Sullivan in \cite{HS}. In 2012, a new class of ideals generated by 2-minors were defined by Qureshi in \cite{Qu}: the \textit{polyomino ideals}. They arise from two-dimensional objects obtained by joining edge by edge unitary squares, called polyominoes. Over the last few years, algebraic properties of polyomino ideals have been investigated, mainly exploiting the combinatorics of the underlying polyomino. One of the most challenging, and still unsolved, algebraic problems on polyominoes is the classification of the prime ones. The fact that a binomial ideal is a prime ideal if and only if it is a toric ideal explains the great interest in prime polyomino ideals. Several steps in this direction have been done, but giving a complete characterization of the prime polyomino ideals does not seem to be an easy task. In  \cite{HM} and \cite{QSS}, the authors prove that simple polyominoes, namely  without holes, are prime. In \cite{HQ} and \cite{Sh1}, a family of prime polyominoes obtained by removing a convex polyomino by a given rectangle was showed. In a more recent paper \cite{MRR}, it is demonstrated that if the polyomino $\MP$ is prime, then it should have no zig-zag walks, and it is conjectured that this is also a sufficient condition for the primality of $\MP$. This conjecture has been verified computationally for all the polyominoes of rank $\leq $ 14. Moreover, in the same work, the authors present a new infinite class of prime: the grid polyominoes. 

Beside the primality, another interesting question concerns the Gr\"obner basis of ideals generated by a subset of $t$-minors, see \cite{Na}, \cite{St1} and \cite{CGG}. As regards polyomino ideals, in \cite{Qu}, the author provides a necessary and sufficient condition for the set of inner 2-minors to be a reduced Gröbner basis of $I_{\MP}$ with respect to two fixed lexicographic monomial orders. Whereas, in \cite{HQS}, Herzog, Qureshi and Shikama show that the ideal of a balanced polyomino has a quadratic Gr\"obner basis with respect to any monomial ideal, that is the ideal is radical. 

In this work, we combine the two above-mentioned questions: we study the primality of the polyomino ideals, by computing their Gr\"obner basis with respect to particular graded reverse lexicographic monomial orders. 
In Section \ref{sec:pre}, we provide the basic definitions regarding polyominoes and their ideals of inner 2-minors. Moreover, we recall the definition given in \cite{Qu} of the lattice ideal associated to a polyomino $\MP$, and we show that it is the ideal quotient of the polyomino ideal $I_{\MP}$ and a monomial. In Section \ref{sec:gb and primality}, we define different graded reverse lexicographic monomial orders and, as in \cite{Qu}, we give a necessary and sufficient condition on $\MP$ for having the set of inner 2-minors as reduced Gr\"obner basis of $I_{\MP}$ (see Proposition \ref{prop:quadraticGB}). Starting from these monomial orders, for any corner $v$ of the polyomino, we define new monomial orders $<_v$ such that the variable $x_v$ is the smallest one with respect to $<_v$. We determine when $I_{\MP}$ admits quadratic Gr\"obner basis with respect to $<_v$ (see Proposition \ref{prop:pik}). In this case, we prove that the ideal is prime (see Theorem \ref{Theo: grid prime}). In the final section of this paper, we apply all the previous results on a class of polyominoes: the \textit{thin polyominoes} (see Definition \ref{def: thin}). We exhibit necessary and sufficient conditions in terms of the geometry of the thin polyomino so that its ideal has a quadratic Gr\"obner basis with respect to some graded reverse lexicographic monomial orders (see Theorem \ref{theo: nec and suff for thin}). As an application we find two subclasses of thin polyominoes that are prime (see Corollary \ref{cor: thin cycle prime} and \ref{cor: gridDeletion}): one is that of \textit{thin cycles} (see Definition \ref{def:thincycle}) with inner intervals of length at least 3, and the other consists of polyominoes obtained from grid polyominoes by the deletion of some cells, that we call \textit{subgrid polyominoes} (see Definition \ref{def: subgrid}). 

\section{Preliminaries and lattice ideals}\label{sec:pre}

Let $a = (i, j), b = (k, \ell) \in \NN^2$, with $i	\leq k$ and $j\leq\ell$, the set $[a, b]=\{(r,s) \in \NN^2 : i\leq r \leq k \text{ and } j \leq s \leq \ell\}$ is called an \textit{interval} of $\NN^2$. If $i<k$ and $j < \ell$, $[a,b]$ is called a \textit{proper interval}, and the elements $a,b,c,d$ are called corners of $[a,b]$, where $c=(i,\ell)$ and $d=(k,j)$. In particular, $a,b$ are called \textit{diagonal corners} and $c,d$ \textit{anti-diagonal corners} of $[a,b]$. The corner $a$ (resp. $c$) is also called the left lower (resp. upper) corner of $[a,b]$, and $d$ (resp. $b$) is the right lower (resp. upper) corner of $[a,b]$. 
A proper interval of the form $C = [a, a + (1, 1)]$ is called a \textit{cell}. Its vertices $V(C)$ are $a, a+(1,0), a+(0,1), a+(1,1)$. The sets $ \{a,a+(1,0)\}, \{a,a+(0,1)\},\{a+(1,0),a+(1,1)\},$ and $\{a+(0,1),a+(1,1)\}$ are called the \textit{edges} of C.
Let $\MP$ be a finite collection of cells of $\NN^2$, and let $C$ and $D$ be two cells of $\MP$. Then $C$ and $D$ are said to be \textit{connected} if there is a sequence of cells $C = C_1,\ldots, C_m = D$ of $\MP$ such that $C_i\cap C_{i+1}$ is an edge of $C_i$
for $i = 1,\ldots, m - 1$. In addition, if $C_i \neq C_j$ for all $i \neq j$, then $C_1,\dots, C_m$ is called a \textit{path} (connecting $C$ and $D$). A collection of cells $\MP$ is called a \textit{polyomino} if any two cells of $\MP$ are connected. We denote by $V(\MP)=\cup _{C\in \MP} V(C)$ the vertex set of $\MP$. 

A polyomino $\mathcal{Q}$ is said to be a \textit{subpolyomino} of a polyomino $\MP$ if each cell belonging to $\MQ$ belongs to $\MP$, and we write $\MQ\subset \MP$.
A proper interval $[a,b]$ is called an \textit{inner interval} of $\MP$ if all cells of $[a,b]$ belong to $\MP$.
We say that a polyomino $\MP$ is \textit{simple} if for any two cells $C$ and $D$ of $\NN^2$ not belonging to $\MP$, there exists a path $C=C_{1},\dots,C_{m}=D$ such that $C_i \notin \MP$ for any $i=1,\dots,m$. 

A finite collection $\MH$ of cells not in $\MP$ is called a \emph{hole} of $\MP$ if any two cells  in $\MH$ are connected through a path of cells in $\MH$, and $\MH$ is maximal with respect to the inclusion. Note that a hole $\MH$ of a polyomino $\MP$ is itself a simple polyomino.

Let $\MP$ be a polyomino. Let $\mathbb{K}$ be a field and $S = \mathbb{K}[x_v \ | \ v \in V(\MP)]$. The binomial $x_a x_b - x_c x_d\in S$ is called an \textit{inner 2-minor} of $\MP$ if $[a,b]$ is an inner interval of $\MP$, where $c,d$ are the anti-diagonal corners of $[a,b]$. We denote by $\MM$ the set of all inner 2-minors of $\MP$. The ideal $I_\MP\subset S$ generated by $\MM$ is called the \textit{polyomino ideal} of $\MP$.

We recall that given a lattice $\Lambda \subseteq \mathbb{Z}^{m\times n}$,  we attach a binomial ideal $I_\Lambda$ called the \emph{lattice ideal} of $\Lambda$ such that
\[
x^\mathbf{a}-x^\mathbf{b} \in I_\Lambda \Leftrightarrow \mathbf{a}-\mathbf{b} \in \Lambda.
\]
We say that a lattice $\Lambda$ is \emph{saturated} if for any $\mathbf{a} \in \mathbb{Z}^{m\times n}$, $c \in \mathbb{Z}$ such that $c\mathbf{a}  \in \Lambda$, we have $\mathbf{a}  \in \Lambda$. It is known that $\Lambda$ is saturated if and only if $I_\Lambda$ is prime. Let $\MP \subseteq [(1,1),(m,n)]$ be a polyomino. Let 
\[
\mathcal{B}=\{\mathbf{e} _{ij} \ : i\in\{1,\ldots,m\}, \ j\in\{1,\ldots, n\}\} 
\] 
be the canonical basis of $\mathbb{Z}^{m\times n}$ and let $\MC=\{C_1,\ldots, C_r\}$ be the set of cells of $\MP$. Let 
$\alpha:\MC\longrightarrow\mathbb{Z}^{m\times n}$ be such that $\alpha(C_k)=\mathbf{c}_k=\mathbf{e}_{ij}+\mathbf{e}_{i+1j+1}-\mathbf{e}_{i+1j}-\mathbf{e}_{ij+1}$, where $(i,j)$ is the lower left corner of the cell $C_k$.

It is known from \cite{ES} that an ideal generated by any set of adjacent 2-minors of a $m \times n$ matrix is a lattice ideal and that its corresponding lattice is saturated. Hence, the lattice $\Lambda=\langle \{\mathbf{c}_k\}_{k=1,\ldots,r} \rangle$ is a saturated lattice, and $I_\Lambda$ is a prime ideal. In addition, it is known from \cite{Qu} that for a collection $\MP$ of cells of $\NN^2$, $I_\MP$ is prime if and only if $I_\MP=I_\Lambda$.
\noindent Moreover,
\begin{lemma}\label{Lemma: lattice}
Let $\MP$ be a collection of cells of $\NN^2$, let $S$ be the polynomial ring associated to $\MP$. Then, there exists a monomial $u \in S$ such that
\[
I_\Lambda=(I_\MP : u).
\]
\end{lemma}
\begin{proof}
$\supseteq)$. Let $u \in S$ be a monomial and let $f \in (I_\MP:u)$. We have that $uf \in I_\MP \subseteq I_\Lambda$. Since $I_\Lambda$ is a prime ideal and $u \notin I_\Lambda$, then $f \in I_\Lambda$.\\
 $\subseteq)$. Let $f_\mathbf{e} = x^{\mathbf{e}^+} - x^{\mathbf{e}^-}$ be a generator of $I_\Lambda$, with
\[
\mathbf{e}=\mathbf{e}^+-\mathbf{e}^-=\sum\limits_{k=1}^r \lambda_k \mathbf{c}_k=\sum\limits_{k=1}^r \left((\lambda_k \mathbf{c}_k)^+ -(\lambda_k \mathbf{c}_k)^-\right) \in \Lambda,
\]
where $\lambda_k \in \ZZ$, $\mathbf{v}^+$ denotes the vector obtained from $\mathbf{v} \in \ZZ^{m \times n}$ by replacing all negative components of $\mathbf{v}$ by zero, and $\mathbf{v}^- = -(\mathbf{v} - \mathbf{v}^+)$. 

Let $\mathbf{v}=\sum\limits_{k=1}^r (\lambda_k \mathbf{c}_k)^+-\mathbf{e}^+=\sum\limits_{k=1}^r (\lambda_k \mathbf{c}_k)^--\mathbf{e}^- $. We have that all the components of $\mathbf{v}$ are non-negative, as for any $k\in \{1,\ldots , r\}$ one has $(\mathbf{c}_{k}^+)_{ij} \geq (\mathbf{c}_k)_{ij}$, for all $1\leq i \leq m$ and $1\leq j \leq n$.
This implies that the monomial $x^\mathbf{v} \in S$ is such that 
\[
x^\mathbf{v}(x^{\mathbf{e}^+}-x^{\mathbf{e}^-})=\prod_{k=1}^r x^{(\lambda_k \mathbf{c}_k)^+} - \prod_{k=1}^r x^{(\lambda_k \mathbf{c}_k)^-}=\sum\limits_{k=1}^r \mu_k  (x^{\mathbf{c}_k^+} -x^{\mathbf{c}_k^-}) \in I_\MP.
\]
If we set $u$ as the least common multiple of the elements $x^\mathbf{v}$ induced by all the generators $f_\mathbf{e}$ of $I_\Lambda$ the assertion follows.
\end{proof}

\section{Quadratic graded reverse lexicographic Gr\"obner basis}\label{sec:gb and primality}

Consider the total orders $<^i$, with $i \in \{1, \dots, 8\}$, on $V(\MP)$ induced by the pairs of arrows displayed in Table \ref{tab: arrows}.

\begin{table}[H]
\begin{tabular}{c | c | c | c | c | c | c | c }
$<^1$ & $<^2$ & $<^3$ & $<^4$ & $<^5$ & $<^6$ & $<^7$ & $<^8$ \\
\hline
$\left( \downarrow \ ,  \rightarrow \right)$  &  $\left( \downarrow \ , \leftarrow\right)$  & $\left(\uparrow \ ,  \leftarrow\right)$ & $\left(\uparrow \ , \rightarrow\right)$ & $ \left(\leftarrow \ , \uparrow\right)$  & $\left( \rightarrow \ , \uparrow\right)$ & $\left( \rightarrow \ , \downarrow\right)$ & $\left( \leftarrow \ ,  \downarrow\right)$ 
\vspace{0.3cm}
\end{tabular}\caption{Pairs of arrows that induce the total orders.}\label{tab: arrows}
\end{table}
Given $a = (a_1, a_2)$ and  $b=(b_1, b_2)$, the horizontal arrows refer to the first coordinates, $a_1$ and $b_1$, while the vertical ones to the second coordinates, $a_2$ and $b_2$. For any pair of arrows, that is for any total order, first, one considers the arrow on the left, and then the other one. Each arrow goes from the minimum to the maximum. For instance, $a <^1 b$ if $a_1 < b_1 $ or $a_1 =b_1 $ and $ a_2 > b_2$. That is, let $a,b,c,d\in V(\MP)$ be as in Figure \ref{Fig:abcd}.
\begin{figure}[H]
\begin{tikzpicture}
\draw (0,0)--(1,0);
\draw (0,0)--(0,1);
\draw (1,1)--(1,0);
\draw (1,1)--(0,1);
\filldraw (0,1) circle (1.5pt) node [anchor=south] {$a$}; 
\filldraw (0,0) circle (1.5pt) node [anchor=north] {$b$}; 
\filldraw (1,1) circle (1.5pt) node [anchor=south] {$c$}; 
\filldraw (1,0) circle (1.5pt) node [anchor=north] {$d$}; 
\end{tikzpicture}
\caption{}\label{Fig:abcd}
\end{figure}
Then it holds $a<^1 b<^1 c<^1 d$. This explains the direction of the arrows.
While $(a_1, a_2) <^5 (b_1, b_2)$ if $a_2 < b_2$ or $a_2 = b_2$ and $a_1>b_1$. In a similar way, one gets all the other orders. 

The total orders $<^i$, with $i \in \{1, \dots, 8\}$, on the vertices of $\MP$ induce in a natural way the graded reverse lexicographic monomial orders $<_{\grevlex}^i$, with $i \in \{1, \dots, 8\}$, on $S=\mathbb{K}[x_v | v \in V(\MP)]$, respectively.

As in \cite[Theorem 4.1]{Qu}, the next proposition gives a necessary and sufficient condition on $\MP$ for having $\MM$ as quadratic reduced Gr\"obner basis of $I_{\MP}$.

From now on, we set $\MO = \{1,3,5,7\}$ and $\ME = \{2,4,6,8\}$.

\begin{proposition}\label{prop:quadraticGB}
Let $\mathcal{P}$ be a polyomino. $\MM$ forms a reduced Gr\"obner basis of $I_{\MP}$ with respect to $<_{\grevlex}^i$, for $i \in \MO$,  if and only if for any two intervals $[a,b]$ and $[b,e]$ of $\MP$, at least one interval between $[a,f]$ and $[a,g]$ is an inner interval of $\MP$, where $f$ and $g$ are the anti-diagonal corners of $[b,e]$. Similarly, $\MM$ forms a reduced Gr\"obner basis of $I_{\MP}$ with respect to $<_{\grevlex}^i$, for $i \in \ME$, if and only if for any two inner intervals $[a,b]$ and $[e,f]$ of $\mathcal{P}$, with $d$ anti-diagonal corner of both the inner intervals, either $a,e$ or $b,f$ are anti-diagonal corners of an inner interval of $\MP$.
\end{proposition}

\begin{proof}
We are going to prove the statement only for $<_{\grevlex}^1$. The others follow in a similar way. The set $\MM$ forms a reduced Gr\"obner basis of $I_{\MP}$ with respect to $<_{\grevlex}^1$ if and only if all $S$-polynomials of inner 2-minors of $I_{\MP}$ reduce to 0. Let $f, g \in \MM$, where $f=x_ax_b-x_cx_d$  is associated to the inner interval $[a,b]$ of $\MP$ and $g=x_px_q-x_rx_s$ is associated to the inner interval $[p,q]$ of $\MP$. In the following, we denote by $S$ the $S$-polynomial between $f$ and $g$ and by $\mathrm{in}(h)$ the leading monomial of a polynomial $h$. We consider the non-trivial cases when $\gcd(\ini(f), \ini(g)) \neq 1$. Moreover, if one of the inner intervals, namely $[a,b]$, is contained in the second one, namely $[p,q]$, $S$ reduces to $0$ since the polyomino ideal is generated by all inner 2-minors. In the following, denote by $<$ the total order $<^1$ on the vertices of $\MP$. Without loss of generality, let $a \leq p$. Therefore, we have to consider the following cases: $a=p$, $b=q$, and $b=p$.\\

  \begin{figure}[H]
 \begin{subfigure}[h]{0.3\textwidth}
\centering
     \resizebox{0.6\textwidth}{!}{
  \begin{tikzpicture}

\draw[thick] (4,2) --  (5,2);
\draw[thick] (4,1) --  (6,1);
\draw[thick] (4,0) --  (6,0);

\draw[thick] (4,0) --  (4,2);
\draw[thick] (5,0) --  (5,2);
\draw[thick] (6,0) --  (6,1);

\fill[fill=gray, fill opacity=0.3] (4,0) -- (5,0)-- (5,2) -- (4,2);
\fill[fill=gray, fill opacity=0.3] (5,0) -- (6,0)-- (6,1) -- (5,1);
\fill[fill=gray, fill opacity=0.5] (4,0) -- (5,0)-- (5,1) -- (4,1);

\filldraw (4,0) circle (1.5pt) node [anchor=north] {$a$}; 
\filldraw (4,1) circle (1.5pt) node [anchor=east] {$c$}; 
\filldraw (5,0) circle (1.5pt) node [anchor=north] {$s$}; 
\filldraw (6,0) circle (1.5pt) node [anchor=north] {$d$}; 
\filldraw (5,1.2) circle (0pt) node [anchor=west] {$e$}; 
\filldraw (6,1) circle (1.5pt) node [anchor=west] {$b$}; 
\filldraw (4,2) circle (1.5pt) node [anchor=south] {$r$}; 
\filldraw (5,2) circle (1.5pt) node [anchor=south] {$q$}; 
\filldraw (5,1) circle (1.5pt);

   \end{tikzpicture}}
   \caption{Case $a=p$.}\label{fig:a=p}
   \end{subfigure}
  ~ 
   \begin{subfigure}[h]{0.3\textwidth}
\centering
     \resizebox{0.6\textwidth}{!}{
  \begin{tikzpicture}
  
  \draw[thick] (0,2) --  (2,2);
\draw[thick] (0,1) --  (2,1);
\draw[thick] (1,0) --  (2,0);

\draw[thick] (0,1) --  (0,2);
\draw[thick] (1,0) --  (1,2);
\draw[thick] (2,0) --  (2,2);

\fill[fill=gray, fill opacity=0.3] (1,0) -- (2,0)-- (2,1) -- (1,1);
\fill[fill=gray, fill opacity=0.3] (0,1) -- (2,1)-- (2,2) -- (0,2);
\fill[fill=gray, fill opacity=0.5] (1,1) -- (2,1)-- (2,2) -- (1,2);

\filldraw (2,0) circle (1.5pt) node [anchor=north] {$s$}; 
\filldraw (1,0) circle (1.5pt) node [anchor=north] {$p$}; 
\filldraw (1,2) circle (1.5pt) node [anchor=south] {$r$}; 
\filldraw (2,1) circle (1.5pt) node [anchor=west] {$d$}; 
\filldraw (2,2) circle (1.5pt) node [anchor=south] {$b$}; 
\filldraw (0,2) circle (1.5pt) node [anchor=south] {$c$}; 
\filldraw (0,1) circle (1.5pt) node [anchor=east] {$a$}; 
\filldraw (1,1) circle (1.5pt) node [anchor=north east] {$e$};

\end{tikzpicture}} \caption{Case $b=q$.}\label{fig:b=q}
\end{subfigure}
  ~ 
   \begin{subfigure}[h]{0.3\textwidth}
\centering
     \resizebox{0.6\textwidth}{!}{
  \begin{tikzpicture}

\draw[thick] (1,2) --  (2,2);
\draw[thick] (0,1) --  (2,1);
\draw[thick] (0,0) --  (1,0);

\draw[thick] (0,0) --  (0,1);
\draw[thick] (1,0) --  (1,2);
\draw[thick] (2,1) --  (2,2);

\draw[dashed] (1,0)--(2,0);
\draw[dashed] (2,0)--(2,1);
\draw[dashed] (0,1)--(0,2);
\draw[dashed] (0,2)--(1,2);

\fill[fill=gray, fill opacity=0.3] (0,0) -- (1,0)-- (1,1) -- (0,1);
\fill[fill=gray, fill opacity=0.3] (1,1) -- (2,1)-- (2,2) -- (1,2);

\filldraw (0,0) circle (1.5pt) node [anchor=north] {$a$}; 
\filldraw (0,1) circle (1.5pt) node [anchor=east] {$c$}; 
\filldraw (1,0) circle (1.5pt) node [anchor=north] {$d$}; 
\filldraw (1,2) circle (1.5pt) node [anchor=south] {$r$}; 
\filldraw (1,1) circle (1.5pt);
\filldraw (1,1.2) circle (0pt) node [anchor=east] {$b$}; 
\filldraw (2,1) circle (1.5pt) node [anchor=west] {$s$}; 
\filldraw (2,2) circle (1.5pt) node [anchor=west] {$q$}; 
\filldraw (0,2) circle (1.5pt) node [anchor=east] {$e$}; 
\filldraw (2,0) circle (1.5pt) node [anchor=north] {$t$};

   \end{tikzpicture}}
   \caption{Case $b=p$.}\label{fig:b=p}
   \end{subfigure}	
   \caption{}\label{fig: cases}
   \end{figure}

\noindent Let $a=p$, that is $f=x_ax_b-x_cx_d$ and $g=x_ax_q-x_rx_s$, and assume $r < c < a < q < s< b< d$ as in Figure \ref{fig:a=p}. We have $S = x_q x_c x_d - x_b x_r x_s$ and $\ini(S) = x_q x_c x_d$. Since $\ini(f_{c,q}) = x_cx_q$, we get
   \[
   S = x_d(x_cx_q - x_rx_e) - x_r(x_sx_b - x_ex_d),
   \]
that is $S$ reduces to 0 with respect to $\MM$.\\
Let $b=q$, and assume $c < a < r < p < b < d < s $ as in Figure \ref{fig:b=q}. We have $S = x_ax_rx_s - x_cx_dx_p$ and $\ini(S) = x_ax_rx_s$. Since $\ini(f_{a,r}) = x_ax_r$, we get
\[
S = x_s(x_ax_r - x_cx_e) - x_c(x_px_d - x_ex_s),
\]
that is $S$ reduces to 0 with respect to $\MM$.\\
Let $b=p$, and assume $c < a < r < b < d < q < s$ as in Figure \ref{fig:b=p}. We have $S=x_a x_r x_s - x_q x_c x_d$ and $\ini(S) = x_ax_rx_s$. If neither $[a,s]$ nor $[a,r]$ is an inner interval of $\MP$, then $S$ does not reduce to 0 with respect to $\MM$ and the Gr\"obner basis is not quadratic. Furthermore, if $[a,s]$ is an inner interval of $\MP$, since $\ini(f_{a,s}) = x_ax_s$, we get
\[
S = x_r(x_ax_s-x_cx_t) -x_c(x_dx_q-x_rx_t).
\]
If $[a,r]$ is an inner interval of $\MP$, since $\ini(f_{a,r}) = x_ax_r$, we get
\[
S = x_s(x_ax_r -x_ex_d) -x_d(x_cx_q - x_ex_s).
\]
It shows that in both situations $S$ reduces to 0 with respect to $\MM$. The latter shows that $S$ reduces to 0 with respect to $\MM$ if and only if either $[a,s]$ or $[a,r]$ is an inner interval of $\MP$ and the thesis follows.
\end{proof}

Let $V(\MP)=\{v_1, \dots, v_n\}$. Given a monomial order $<$ such that we have 
\[
 x_{v_1}<x_{v_2}<\cdots<x_{v_n},
\]
we define by $<_v$, with $v = v_k \in V(\MP)$, the following monomial order:
\[
 x_{v_k}<x_{v_{k+1}}<\cdots<x_{v_n}<x_{v_1}<x_{v_2}<\cdots<x_{v_{k-1}}.
\]
From now on, we will denote $(<^i_{\grevlex})_v$ by $<_v^i$, for any $i \in \{1, \dots, 8\}$. \\

\begin{definition}\label{def: cond Pi}
Let $\MP$ be a polyomino and let $v \in V(\MP)$. We say that $v$ satisfies the condition $\pi_1$ if it fulfils at least one of the following conditions:
\begin{enumerate}[(I)]
\item There exist two inner intervals $I=[a,b]$ and $J=[b,q]$ of $\MP$, with $v$ upper left corner of $I$, and $s$ the lower right corner of $J$, such that $[v,q]$ is inner interval of $\MP$, whereas the interval $[a,s]$ is not (see Table \ref{Tab: Cond Pi}, Case $\pi_1$ (I)).
\item There exist two inner intervals $K=[a,b]$ and $L=[p,q]$, with $v$ lower right corner of $K$ and upper left corner of $L$, such that the interval having $b$ and $q$ as anti-diagonal corners is inner interval of $\MP$, whereas the interval having $a$ and $p$ as anti-diagonal corners is not (see Table \ref{Tab: Cond Pi}, Case $\pi_1$ (II)).
\end{enumerate}
\end{definition}

In a similar way, one can define $v$ satisfying the condition $\pi_i$, for $i \in \{2, \dots, 8\}$, if it fulfils at least one of the cases (I) and (II) displayed in Table \ref{Tab: Cond Pi}. 

\begin{table}[H]
 \centering
 \resizebox{0.8\textwidth}{!}{
   \begin{tabular}{|c|c|c|c|c|}
   \hline
   &$\pi_1$     &$\pi_2$           &$\pi_3$    &$\pi_4$\\
   \hline
 \parI &\oneI &\twoI &\threeI &\fourI \\
   \hline
\parII &\oneII &\twoII &\threeII &\fourII		\\
\hline
\hline
   &$\pi_5$  &$\pi_6$ &$\pi_7$ &$\pi_8$ \\
   \hline
 \parI &\fiveI &\sixI &\sevenI &\eightI \\
   \hline
\parII &\fiveII &\sixII &\sevenII &\eightII \\
   \hline
   \end{tabular}}
   \ \\ \ \\ \caption{Conditions $\pi_i$, for $i=1, \dots, 8$.}\label{Tab: Cond Pi}
   \end{table}

\begin{proposition}\label{prop:pik}
Let $\MP$ be a polyomino such that $I_{\MP}$ has $\MM$ as reduced Gr\"obner basis with respect to $<^i_{\grevlex}$, with $i \in \MO$ ($i \in \ME$, respectively). If $v \in  V(\MP)$ does not satisfy $\pi_k$ for some $k \in \MO$ ($k \in \ME$, respectively), then $\MM$ forms a reduced Gr\"obner basis of $I_\MP$ with respect to $<_v^k$.
\end{proposition}

\begin{proof}
Assume that $\MM$ forms a reduced Gr\"obner basis of $I_\MP$ with respect to $<^i_{\grevlex}$, with $i \in \MO$. Let $f=x_ax_b-x_cx_d$ and $g=x_px_q-x_rx_s$ be associated to the inner interval $[a,b]$ and $[p,q]$ of $\MP$, respectively. Let $v \in V(\MP)$.  We have to show that for each pair of inner 2-minors, $f$ and $g$, the corresponding $S$-polynomial reduces to $0$ with respect to a fixed monomial order $<_v^i$, with $i \in \MO$. In the following, we denote by $S$ the $S$-polynomial between $f$ and $g$, by $\mathrm{in}(h)$ the leading monomial of a polynomial $h$, and by $f_{m,n}$ the inner 2-minor associated to the inner interval $[m,n]$ of $\MP$.

We leave to the reader the trivial cases $\{a,b,c,d\}\cap\{p,q,r,s\}=\varnothing$, and  $|\{a,b,c,d\}\cap\{p,q,r,s\}|=2$ where $S$ reduces to $0$ since the polyomino ideal is generated by all inner 2-minors.

Note that if, for all vertices $w\in\{a,b,c,d,p,q,r,s\}$ and a monomial order $<^i_{\grevlex}$, for some $i \in \MO$, it holds $x_w <_v^i x_v$ or $x_v <_v^i x_w$, then $S$ reduces to 0 with respect to $<_v^i$, since it reduces to 0 with respect to $<^i_{\grevlex}$. 

If one of the inner intervals, namely $[a,b]$, is contained in the second one, namely $[p,q]$, $S$ reduces to $0$  since the polyomino ideal is generated by all inner 2-minors. In the following, denote by $<$ the total order $<^1$ on the vertices of $\MP$. Without loss of generality, let $a \leq p$. Therefore, we have to consider the following cases: 
\[
a = p, \hspace{2cm} b,d \in \{p,q,r,s\}, \hspace{2cm} c \in \{p,r\}.
\]
If $v$ does not satisfy the condition $\pi_k$, for some $k \in \MO$, we fix the monomial order $<_v^k$. Without loss of generality, assume $k=1$.

Let $a=p$, that is $f=x_ax_b-x_cx_d$ and $g=x_ax_q-x_rx_s$, and $r<c<a<q<s<b<d$ as in Figure \ref{fig: a=p}.

  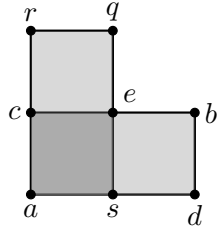
\begin{figure}[H]
  \centering
    \resizebox{0.25\textwidth}{!}{
  \begin{tikzpicture}

\draw[thick] (4,2) --  (5,2);
\draw[thick] (4,1) --  (6,1);
\draw[thick] (4,0) --  (6,0);

\draw[thick] (4,0) --  (4,2);
\draw[thick] (5,0) --  (5,2);
\draw[thick] (6,0) --  (6,1);

\fill[fill=gray, fill opacity=0.3] (4,0) -- (5,0)-- (5,2) -- (4,2);
\fill[fill=gray, fill opacity=0.3] (5,0) -- (6,0)-- (6,1) -- (5,1);
\fill[fill=gray, fill opacity=0.5] (4,0) -- (5,0)-- (5,1) -- (4,1);

\filldraw (4,0) circle (1.5pt) node [anchor=north] {$a$}; 
\filldraw (4,1) circle (1.5pt) node [anchor=east] {$c$}; 
\filldraw (5,0) circle (1.5pt) node [anchor=north] {$s$}; 
\filldraw (6,0) circle (1.5pt) node [anchor=north] {$d$}; 
\filldraw (5,1.2) circle (0pt) node [anchor=west] {$e$}; 
\filldraw (6,1) circle (1.5pt) node [anchor=west] {$b$}; 
\filldraw (4,2) circle (1.5pt) node [anchor=south] {$r$}; 
\filldraw (5,2) circle (1.5pt) node [anchor=south] {$q$}; 
\filldraw (5,1) circle (1.5pt);

   \end{tikzpicture}}
   \caption{Case $a=p$.}\label{fig: a=p}
   \end{figure} 
\noindent We start by observing that if $r < v \leq b$, then $\gcd(\mathrm{in} (f), \mathrm{in} (g))=1$. In the other cases, we have $S= x_rx_sx_b - x_cx_dx_q$. If $b <v \leq d$, then $\mathrm{in}(S)= x_rx_sx_b$. Since $\ini(f_{s,b}) = x_sx_b$, then
\[
S= x_r(x_sx_b - x_ex_d) -x_d(x_cx_q-x_rx_e),
\]
that is $S$ reduces to 0 with respect to the inner 2-minors $f_{s,b}$ and $f_{c,q}$. If $v=r$, then $\ini(S) = x_cx_dx_q$. Since $\ini(f_{c,q}) = x_cx_q$, then
\[
S = -x_d(x_cx_q-x_rx_e) +x_r(x_sx_b-x_ex_d),
\]
that is $S$ reduces to 0 with respect to the inner 2-minors $f_{c,q}$ and $f_{s,b}$.

Let $b=p$, that is $f= x_ax_b - x_cx_d$ and $g=x_bx_q - x_rx_s$, and$c<a<r<b<d<q<s$, as in Figure \ref{Fig: b=p}. 

 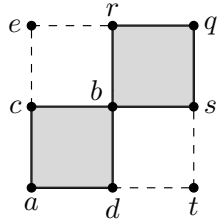
\begin{figure}[H]
  \centering
    \resizebox{0.25\textwidth}{!}{
  \begin{tikzpicture}

\draw[thick] (1,2) --  (2,2);
\draw[thick] (0,1) --  (2,1);
\draw[thick] (0,0) --  (1,0);

\draw[thick] (0,0) --  (0,1);
\draw[thick] (1,0) --  (1,2);
\draw[thick] (2,1) --  (2,2);

\draw[dashed] (1,0)--(2,0);
\draw[dashed] (2,0)--(2,1);
\draw[dashed] (0,1)--(0,2);
\draw[dashed] (0,2)--(1,2);

\fill[fill=gray, fill opacity=0.3] (0,0) -- (1,0)-- (1,1) -- (0,1);
\fill[fill=gray, fill opacity=0.3] (1,1) -- (2,1)-- (2,2) -- (1,2);

\filldraw (0,0) circle (1.5pt) node [anchor=north] {$a$}; 
\filldraw (0,1) circle (1.5pt) node [anchor=east] {$c$}; 
\filldraw (1,0) circle (1.5pt) node [anchor=north] {$d$}; 
\filldraw (1,2) circle (1.5pt) node [anchor=south] {$r$}; 
\filldraw (1,1) circle (1.5pt);
\filldraw (1,1.2) circle (0pt) node [anchor=east] {$b$}; 
\filldraw (2,1) circle (1.5pt) node [anchor=west] {$s$}; 
\filldraw (2,2) circle (1.5pt) node [anchor=west] {$q$}; 
\filldraw (0,2) circle (1.5pt) node [anchor=east] {$e$}; 
\filldraw (2,0) circle (1.5pt) node [anchor=north] {$t$};

   \end{tikzpicture}}
   \caption{Case $b=p$.}\label{Fig: b=p}
   \end{figure}
   
\noindent If $c < v \leq q  $, then $\gcd(\mathrm{in} (f), \mathrm{in} (g)) =1$.
In the other cases, we have  $S = x_ax_rx_s - x_qx_cx_d$.  If $q < v \leq s$, then $\mathrm{in}(S)=x_qx_cx_d$. By hypothesis, $\MM$ forms a reduced Gr\"obner basis of $I_{\MP}$ with respect to $<_{\grevlex}^i$ with $i\in \MO$, hence, from Proposition \ref{prop:quadraticGB}, either $[c,q]$ or $[d,q]$ is an inner interval of $\MP$, with $\mathrm{in}(f_{c,q}) = x_c x_q$ and  $\mathrm{in}(f_{d,q}) = x_d x_q$, and then 
\[
S = x_d(x_cx_q - x_ex_s) - x_s (x_ax_r -x_ex_d)
\]
or 
\[
S= -x_c(x_dx_q -x_rx_t) + x_r( x_ax_s-x_cx_t),
\]
that is $S$ reduces to 0 with respect to the inner 2-minors either $f_{c,q}$ and $f_{a,r}$ or $f_{d,q}$ and $f_{a,s}$. If $v=c$, then $\mathrm{in} (S) = x_ax_rx_s$. By hypothesis, either $[a,r]$ or $[a,s]$ is an inner interval of $\MP$ , with $\mathrm{in}(f_{a,r}) = x_ex_d$ and $\mathrm{in}(f_{a,s}) = x_ax_s$. If $[a,r]$ is an inner interval of $\MP$, but $[a,s]$ is not, then $v$ satisfies the condition $\pi_1$, so we have not to consider this case. Whereas, if $[a,s]$ is an inner interval, since $\ini(f_{a,s})=x_ax_s$, then 
\[
S = x_r(x_ax_s-x_cx_t) -x_c(x_dx_q-x_rx_t),
\]
it follows that $S$ reduces to 0. \\
Note that when $v=c$, if $[a,r]$ is an inner interval of $\MP$, but $[a,s]$ is not, that is $v$ satisfies $\pi_1$, in particular the condition $\pi_1$ (I), then $S$ does not reduce to 0 with respect to $\MM$ and $<_v^1$. In fact, $\mathrm{in}(S)= x_ax_rx_s$, but the monomials $x_ax_r$, $x_ax_s$, and $x_rx_s$ are not leading monomials of any inner 2-minor of $\MP$. This situation justifies the hypothesis $v$ not satisfying the condition $\pi_1$.

Let $b=r$, that is that is $f=x_ax_b - x_cx_d$ and $g = x_px_q -x_bx_s$. We have to distinguish two different situations: $p<d$ (see Figure \ref{fig:b=r} (A)) or $p>d$ (see Figure \ref{fig:b=r} (B)).

 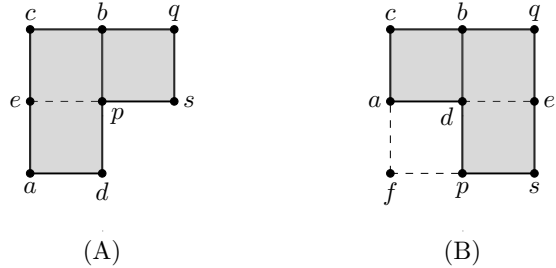
\begin{figure}[H]
  \centering
    \resizebox{0.6\textwidth}{!}{
  \begin{tikzpicture}

\draw[thick] (0,2) --  (2,2);
\draw[thick] (1,1) --  (2,1);
\draw[thick] (0,0) --  (1,0);

\draw[thick] (0,0) --  (0,2);
\draw[thick] (1,0) --  (1,2);
\draw[thick] (2,1) --  (2,2);

\draw[dashed] (0,1)--(1,1);

\fill[fill=gray, fill opacity=0.3] (0,0) -- (1,0)-- (1,2) -- (0,2);
\fill[fill=gray, fill opacity=0.3] (1,1) -- (2,1)-- (2,2) -- (1,2);

\filldraw (0,0) circle (1.5pt) node [anchor=north] {$a$}; 
\filldraw (0,1) circle (1.5pt) node [anchor=east] {$e$}; 
\filldraw (1,0) circle (1.5pt) node [anchor=north] {$d$}; 
\filldraw (1,0.8) circle (0pt) node [anchor=west] {$p$}; 
\filldraw (1,1) circle (1.5pt);
\filldraw (2,1) circle (1.5pt) node [anchor=west] {$s$}; 
\filldraw (0,2) circle (1.5pt) node [anchor=south] {$c$}; 
\filldraw (1,2) circle (1.5pt) node [anchor=south] {$b$}; 
\filldraw (2,2) circle (1.5pt) node [anchor=south] {$q$}; 

\filldraw (1,-0.8) circle (0pt) node [anchor=north] {(A)}; 
\filldraw (1+5,-0.8) circle (0pt) node [anchor=north] {(B)};

\draw[thick] (5,2) --  (7,2);
\draw[thick] (5,1) --  (6,1);
\draw[thick] (6,0) --  (7,0);

\draw[thick] (5,1) --  (5,2);
\draw[thick] (6,0) --  (6,2);
\draw[thick] (7,0) --  (7,2);

\draw[dashed] (1+5,1)--(2+5,1);
\draw[dashed] (0+5,0)--(0+5,1);
\draw[dashed] (0+5,0)--(1+5,0);

\fill[fill=gray, fill opacity=0.3] (5,1) -- (5,2)-- (6,2) -- (6,1);
\fill[fill=gray, fill opacity=0.3] (6,0) -- (6,2)-- (7,2) -- (7,0);

\filldraw (5,1) circle (1.5pt) node [anchor=east] {$a$}; 
\filldraw (5,2) circle (1.5pt) node [anchor=south] {$c$}; 
\filldraw (6,0) circle (1.5pt) node [anchor=north] {$p$}; 
\filldraw (6,0.8) circle (0pt) node [anchor=east] {$d$}; 
\filldraw (6,1) circle (1.5pt);
\filldraw (2+5,1) circle (1.5pt) node [anchor=west] {$e$}; 
\filldraw (6,2) circle (1.5pt) node [anchor=south] {$b$}; 
\filldraw (7,0) circle (1.5pt) node [anchor=north] {$s$}; 
\filldraw (7,2) circle (1.5pt) node [anchor=south] {$q$}; 
\filldraw (0+5,0) circle (1.5pt) node [anchor=north] {$f$};

   \end{tikzpicture}}
   \caption{Case $b=r$.}\label{fig:b=r}
   \end{figure} 
   
\noindent Assume $p<d$, then $c<a<b<p<d<q<s$, as in Figure \ref{fig:b=r} (A). If $c \leq v \leq b$ or $q < v \leq s$, then $\gcd(\ini(f), \ini(g)) =1$. In the other cases, $S= x_ax_px_q - x_cx_dx_s$. If $b < v \leq p$ or $d < v \leq q$, then $\ini(S) = x_cx_dx_s$ and $\ini(f_{e,q}) = x_cx_s$.  If $p < v \leq d$, then $\ini(S) =  x_ax_px_q$ and $\ini(f_{a,p}) = x_ax_p$. Therefore,
\[
S= x_d(x_ex_q-x_cx_s) + x_q(x_ax_p-x_ex_d),
\]
that is $S$ reduces to 0 in all of these cases.

\noindent Assume $p>d$, then $c<a<b<d<p<q<s$, as in Figure \ref{fig:b=r} (B). If $c \leq v \leq b$ or $q < v \leq s$, then $\gcd(\ini(f), \ini(g)) =1$. In the other cases, we have $S= x_ax_px_q - x_cx_dx_s$. If $b < v \leq d$, then $\ini(S) = x_ax_px_q$. By hypothesis, $v$ does not satisfy the condition $\pi_1$, hence $[f,d]$ is an inner interval of $\MP$. Since $\ini(f_{f,d})=x_ax_p$, then
\[
S = -x_q(x_fx_d - x_ax_p) + x_d(x_fx_q -x_cx_s),
\]
that is $S$ reduces to 0. If $ d < v \leq q$, then $\ini(S) = x_cx_dx_s$. Since $\ini(f_{p,e}) = x_dx_s$, it follows
\[
S = x_c(x_px_e-x_dx_s) + x_p(x_ax_q-x_cx_e),
\]
that is $S$ reduces to 0.

Note that when $b < v \leq d$, if $[f,d]$ is not an inner interval of $\MP$, then $v$ satisfies $\pi_1$, in particular the condition $\pi_1$ (II). In this case, $S$ does not reduce to 0 with respect to $\MM$ and $<_v^1$. In fact, $\mathrm{in}(S)= x_ax_px_q$, but the monomials $x_ax_p$, $x_ax_q$, and $x_px_q$ are not leading monomials of any inner 2-minor of $\MP$. This situation justifies, once again, the hypothesis $v$ not satisfying the condition $\pi_1$.

Let $d=q$, that is $f=x_ax_b - x_cx_d$ and $g=x_px_d -x_rx_s$, and $c<a<r<p<b<d<s$, as showed in Figure \ref{fig:d=q}. 
 \begin{figure}[H]
  \centering
    \resizebox{0.25\textwidth}{!}{
  \begin{tikzpicture}

\draw[thick] (0,2) --  (2,2);
\draw[thick] (0,1) --  (2,1);
\draw[thick] (1,0) --  (2,0);

\draw[thick] (0,1) --  (0,2);
\draw[thick] (1,0) --  (1,1);
\draw[thick] (2,0) --  (2,2);

\draw[dashed] (1,1)-- (1,2);
\draw[dashed] (0,0)-- (1,0);
\draw[dashed] (0,0)-- (0,1);

\fill[fill=gray, fill opacity=0.3] (1,0) -- (2,0)-- (2,1) -- (1,1);
\fill[fill=gray, fill opacity=0.3] (0,1) -- (2,1)-- (2,2) -- (0,2);

\filldraw (1,2) circle (1.5pt) node [anchor=south] {$e$}; 
\filldraw (0,1) circle (1.5pt) node [anchor=east] {$a$}; 
\filldraw (1,0) circle (1.5pt) node [anchor=north] {$p$}; 
\filldraw (1,1) circle (1.5pt);
\filldraw (1,1.2) circle (0pt) node [anchor=east] {$r$}; 
\filldraw (2,1) circle (1.5pt) node [anchor=west] {$d$}; 
\filldraw (2,2) circle (1.5pt) node [anchor=south] {$b$}; 
\filldraw (2,0) circle (1.5pt) node [anchor=north] {$s$}; 
\filldraw (0,2) circle (1.5pt) node [anchor=south] {$c$}; 
\filldraw (0,0) circle (1.5pt) node [anchor=north] {$f$}; 

   \end{tikzpicture}}\caption{Case $d=q$.}\label{fig:d=q}
   \end{figure}
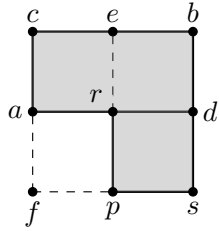
   
   \noindent If either $v=c$ or $r< v \leq s$, then $\gcd(\ini(f), \ini(g)) =1$. In the other cases, we have $S=x_ax_bx_p-x_cx_rx_s$. If $c < v \leq a$, then $\ini(S)=x_cx_rx_s$. Since $\ini(f_{c,r}) = x_cx_r$, then 
   \[
   S= x_s(x_ax_e-x_cx_r) +x_a(x_px_b - x_sx_e),
   \]
   that is $S$ reduces to 0. If $a < v \leq r$, then $\ini(S) = x_ax_bx_p$. By hypothesis, $v$ does not satisfy $\pi_1$, that is $[f,r]$ is an inner interval of $\MP$. Since $\ini(f_{f,r})=x_ax_p$, then 
   \[
   S = -x_b(x_fx_r - x_ax_p) + x_r(x_fx_b - x_cx_s),
   \]
      that is $S$ reduces to 0.

Let $c=r$, that is $f = x_ax_b - x_cx_d$ and $g = x_px_q - x_cx_s$, and $c<p<a<b<d<q<s$, as showed in Figure \ref{fig:c=r}.
   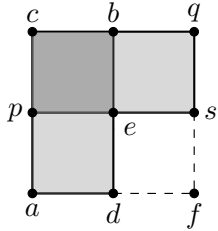
\begin{figure}[H]
  \centering
    \resizebox{0.25\textwidth}{!}{
  \begin{tikzpicture}

\draw[thick] (0,2) --  (2,2);
\draw[thick] (0,1) --  (2,1);
\draw[thick] (1,0) --  (0,0);

\draw[thick] (0,0) --  (0,2);
\draw[thick] (1,0) --  (1,2);
\draw[thick] (2,1) --  (2,2);

\draw[dashed] (1,0) -- (2,0) -- (2,1);

\fill[fill=gray, fill opacity=0.3] (0,0) -- (1,0)-- (1,1) -- (0,1);
\fill[fill=gray, fill opacity=0.3] (0,1) -- (2,1)-- (2,2) -- (0,2);
\fill[fill=gray, fill opacity=0.5] (0,1) -- (1,1)-- (1,2) -- (0,2);

\filldraw (0,0) circle (1.5pt) node [anchor=north] {$a$}; 
\filldraw (0,1) circle (1.5pt) node [anchor=east] {$p$}; 
\filldraw (1,0) circle (1.5pt) node [anchor=north] {$d$}; 
\filldraw (1,2) circle (1.5pt) node [anchor=south] {$b$}; 
\filldraw (2,1) circle (1.5pt) node [anchor=west] {$s$}; 
\filldraw (2,2) circle (1.5pt) node [anchor=south] {$q$}; 
\filldraw (0,2) circle (1.5pt) node [anchor=south] {$c$}; 
\filldraw (1,0.8) circle (0pt) node [anchor=west] {$e$}; 
\filldraw (2,0) circle (1.5pt) node [anchor=north] {$f$}; 
\filldraw (1,1) circle (1.5pt);
   \end{tikzpicture}}
   \caption{Case $c=r$.}\label{fig:c=r}
   \end{figure}

\noindent If either $v=c$ or $b < v \leq s$ , then $\gcd(\ini(f), \ini(g)) =1$. In the other cases, we have  $S=x_ax_bx_s - x_dx_px_q$. If $c < v \leq p$, $\ini(S) = x_ax_bx_s$.  Since $v$ does not satisfy $\pi_1$, then $[a,s]$ is an inner interval of $\MP$ and $\ini(f_{a,s}) = x_ax_s$. Therefore, 
\[
S= x_b(x_ax_s - x_px_f) -x_p(x_dx_q - x_bx_f),
\]
that is $S$ reduces to 0. If $p < v \leq b$, then $\ini(S) = x_dx_px_q$ and $\ini(f_{a,e}) = x_px_d$. Therefore, 
\[
S= x_q(x_ax_e- x_px_d) - x_a(x_ex_q-x_bx_s),
\]
that is $S$ reduces to 0. 
For the sake of brevity, we leave to readers to check, in a similar way, that if $b \in \{q,s\}$, $d \in \{p,r,s\}$, and $c=p$, then all the $S$-polynomials reduce to 0. Moreover, for no one of the corners $v$ in these cases it needs to require the hypothesis that $v$ does not satisfy the condition $\pi_1$. 

\end{proof}

We now prove the main theorem of this section.


\begin{theorem}\label{Theo: grid prime}
Let $\MP$ be a polyomino such that $I_{\MP}$ has $\MM$ as reduced Gr\"obner basis of $I_\MP$ with respect to $<^i_{\grevlex}$, with $i \in \MO$ ($i \in \ME$, respectively). If, for all $v \in  V(\MP)$, there exists a $k_v \in  \MO$ ($k_v \in \ME$, respectively) such that $v$ does not satisfy $\pi_{k_v}$, then
\begin{enumerate}
\item $\MM$ forms a reduced Gr\"obner basis with respect to $<_v^{k_v}$, for all $v \in V(\MP)$;
\item  $I_{\MP}$ is prime. 
\end{enumerate}
\end{theorem}

\begin{proof}
(1) It is an immediate consequence of Proposition \ref{prop:pik}.\\
(2) Fix $v \in V(\MP)$. By (1),  let $<_v$ denote the monomial order for which $\MM$ forms a reduced Gr\"obner basis of $I_\MP$. By \cite[Lemma 12.1]{St}, the reduced Gr\"obner basis of $(I_{\MP} : x_v)$ with respect to $<_v$ is given by
\[
\{ f \in \MM \mid x_v \text{ does not divide } f\} \cup \{ f/x_v \mid  f \in \MM \text{ and } x_v \text{ divides } f\}.
\] 
Since all $f \in \MM$ are not divisible by $x_v$, the reduced Gr\"obner basis of $(I_{\MP} : x_v)$ with respect to $<_v$ is $\MM$. Therefore $(I_{\MP}: x_v) = I_{\MP}$, for all $x_v \in V(\MP)$. It follows that $(I_\MP : u) = I_\MP$ for any monomial $u \in S$. By Lemma \ref{Lemma: lattice}, we have that there exists a monomial $u \in S$ such that $I_{\Lambda} = (I_{\MP} : u)$. Then 
\[
I_{\Lambda} = (I_{\MP} : u) = I_{\MP}.
\]
It follows that $I_{\MP}$ coincides with the lattice ideal $I_{\Lambda}$, which is prime. Therefore, $I_{\MP}$ is a prime ideal, as well. 
\end{proof}

\section{Thin polyominoes}\label{sec: thin}
In this section, we introduce the class of thin polyominoes and we rephrase the geometric condition for the quadratic  Gr\"obner basis of $I_\MP$ in Proposition \ref{prop:quadraticGB} in terms of some subpolyominoes of the thin polyomino $\MP$. Thanks to the above interpretation, we find two new classes of thin polyominoes having a prime polyomino ideal: the thin cycle with no maximal inner interval of length 2 and the subgrid polyominoes.

\begin{definition}	\label{def: thin}
Let $\MP$ be a polyomino. We say that $\MP$ is \emph{thin} if $\MP$ does not have the polyomino $\QQ$ in Figure \ref{fig:square} as a subpolyomino.
\begin{figure}[H]
  \resizebox{0.15\textwidth}{!}{
  \begin{tikzpicture}
  \draw (0,0)--(2,0);
   \draw (0,1)--(2,1);
   \draw (0,2)--(2,2);
   
   \draw (0,0)--(0,2);
   \draw (1,0)--(1,2);
    \draw (2,0)--(2,2);
\end{tikzpicture}}\caption{The polyomino $\QQ$}\label{fig:square}
\end{figure}
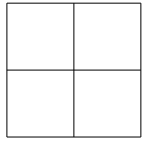 
\end{definition}

\begin{theorem}\label{Theo: thin poly GB}
Let $\MP$ be a thin polyomino such that $\MM$ forms a reduced Gr\"obner basis of $I_\MP$ with respect to $<_{\grevlex}^{i}$ for $i \in \MO$ (for $i \in \ME$, respectively). Then, for any $v\in V(\MP)$, there exists $k\in \MO$ ($k \in \ME$, respectively) such that $\MM$ forms a reduced Gr\"obner basis of $I_\MP$ with respect to $<_v^k$.
\end{theorem}
\begin{proof}
Assume that $\MM$ forms a reduced Gr\"obner basis of $I_\MP$ with respect to $<_{\grevlex}^{i}$, with $i \in \MO$.
Let $v\in V(\MP)$. From Proposition \ref{prop:pik} it suffices to show that there exists $k \in \MO$ such that $v$ does not satisfy $\pi_k$. 

We claim that $v$ can not satisfy simultaneously $\pi_1$ and $\pi_3$. In fact, if $v$ satisfies simultaneously $\pi_1$ and $\pi_3$, then there exist four cells $C,D,E,F$ of $\MP$ such that $C\cap D\cap E \cap F=\{v\}$. From Table \ref{Tab: Cond Pi}, if $v$ satisfies $\pi_1$ then there exist two cells  $C,D$ of $\MP$ such that $v$ is simultaneously the lower left corner of $C$ and the upper left corner of $D$, while if $\PP$ satisfies $\pi_3$ then there exist two cells  $E,F$ of $\MP$ such that $v$ is simultaneously the lower right corner of $E$ and the upper right corner of $F$. Since $v$ satisfies simultaneously $\pi_1$ and $\pi_3$, the cells $C,D,E,F$ are the ones desired. This implies that the polyomino $\QQ$ in Figure \ref{fig:square} is a subpolyomino of $\MP$ and then $\PP$ is not thin, which is a contradiction. It follows that there exists at least a $k \in \MO$ such that $v$ does not satisfy $\pi_k$, as desired.
\end{proof}

\begin{corollary}\label{Cor: quadratic GB implies prime}
Let $\MP$ be a thin polyomino such that $\MM$ forms a reduced Gr\"obner basis of $I_\MP$ with respect to $<_{\grevlex}^{i}$ for $i \in \{1,\ldots,8\}$. Then $I_{\MP}$ is prime. 
\end{corollary}
\begin{proof}
By Theorem \ref{Theo: thin poly GB}, for any $v \in V(\MP)$, $\MM$ forms a reduced Gr\"obner basis of $I_{\MP}$ with respect to $<_v^k$, for some $k \in \{1, \dots, 8\}$. By Theorem \ref{Theo: grid prime}, it follows that $I_{\MP}$ is prime. 
\end{proof}

\begin{theorem}\label{theo: nec and suff for thin}
Let $\MP$ be a thin polyomino. The following facts are equivalent:
\begin{enumerate}
\item $\MM$ forms a reduced Gr\"obner basis of $I_\MP$ with respect to $<_{\grevlex}^{i}$ for $i \in \mathcal{O}$ ($i \in \mathcal{E}$, respectively);
\item there are no cells $C,D \not \in \MP$ and $E,F \in \MP$ such that $C \cap D \cap E \cap  F \neq \emptyset$ as in Figure \ref{fig:CD} (a) (Figure \ref{fig:CD} (b), respectively) and the polyominoes in Figure \ref{Fig: no in prime thin} (i) and (ii) (in Figure \ref{Fig: no in prime thin} (iii) and (iv), respectively) are not subpolyominoes of $\MP$.
\begin{figure}[H]
\resizebox{0.5\textwidth}{!}{
\begin{tikzpicture}
\draw (0,0)--(0,2);
\draw (0,0)--(2,0);
\draw (0,2)--(2,2);
\draw (2,0)--(2,2);
\draw(1,0)--(1,2);
\draw (0,1)--(2,1);
\filldraw[gray,fill opacity=0.5] (0,0)--(1,0)--(1,1)--(0,1);
\filldraw[gray, fill opacity=0.5] (1,1)--(2,1)--(2,2)--(1,2);
\node at (0.5,1.5){$C$};
\node at (1.5,0.5){$D$};
\node at (0.5,0.5){$E$};
\node at (1.5,1.5){$F$};
\node (d) at (1,-0.5)[below] {\large{$\mathrm{(a)}$}};

\draw (0+5,0)--(0+5,2);
\draw (0+5,0)--(2+5,0);
\draw (0+5,2)--(2+5,2);
\draw (2+5,0)--(2+5,2);
\draw(1+5,0)--(1+5,2);
\draw (0+5,1)--(2+5,1);
\filldraw[gray,fill opacity=0.5] (0+5,1)--(1+5,1)--(1+5,2)--(0+5,2);
\filldraw[gray, fill opacity=0.5] (1+5,0)--(2+5,0)--(2+5,1)--(1+5,1);
\node at (0.5+5,1.5){$E$};
\node at (1.5+5,0.5){$F$};
\node at (0.5+5,0.5){$C$};
\node at (1.5+5,1.5){$D$};
\node (d) at (1+5,-0.5)[below] {\large{$\mathrm{(b)}$}};
\end{tikzpicture}}\caption{}\label{fig:CD}
\end{figure}

\begin{figure}[H]
\resizebox{0.9\textwidth}{!}{

\begin{tikzpicture}
\draw (0,0)--(2,0);
\draw (0,1)--(3,1);
\draw (1,2)--(3,2);

\draw (0,0)--(0,1);
\draw (1,0)--(1,2);
\draw (2,0)--(2,2);
\draw (3,1)--(3,2);

\node (d) at (1.5,-0.5)[below] {\Large{$\mathrm{(i)}$}};


\draw (0+5,0)--(1+5,0);
\draw (0+5,1)--(2+5,1);
\draw (0+5,2)--(2+5,2);
\draw (1+5,3)--(2+5,3);

\draw (0+5,0)--(0+5,2);
\draw (1+5,0)--(1+5,3);
\draw (2+5,1)--(2+5,3);

\node (d) at (1+5,-0.5)[below] {\Large{$\mathrm{(ii)}$}};

\draw (1+10,0)--(3+10,0);
\draw (0+10,1)--(3+10,1);
\draw (0+10,2)--(2+10,2);

\draw (0+10,1)--(0+10,2);
\draw (1+10,0)--(1+10,2);
\draw (2+10,0)--(2+10,2);
\draw (3+10,0)--(3+10,1);

\node (d) at (1.5+10,-0.5)[below] {\Large{$\mathrm{(iii)}$}};

\draw (0+15,1)--(0+15,3);
\draw (1+15,0)--(1+15,3);
\draw (2+15,0)--(2+15,2);

\draw (1+15,0)--(2+15,0);
\draw (0+15,1)--(2+15,1);
\draw (0+15,2)--(2+15,2);
\draw (0+15,3)--(1+15,3);

\node (d) at (1+15,-0.5)[below] {\Large{$\mathrm{(iv)}$}};

\end{tikzpicture}}\caption{}\label{Fig: no in prime thin}
\end{figure}
\end{enumerate}
\end{theorem}
\begin{proof}
We prove the equivalent statements for $<^{i}_{\mathrm{grevlex}}$, with $i\in \mathcal{O}$. The case $<^{i}_{\mathrm{grevlex}}$ for $i\in \mathcal{E}$ can be done similarly. \\
(1) $\Rightarrow$ (2). 
Firstly, let $E, F$ be two cells of $\MP$ as in Figure \ref{fig:CD} (a). Since,  by hypothesis, $\MM$ is a quadratic Gr\"obner basis, by Proposition \ref{prop:quadraticGB}, at least one cell between $C$ and $D$ must be a cell of $P$. That is the situation displayed in Figure \ref{fig:CD} (a) is not possible. 
Secondly, assume, by contradiction, that the polyominoes in Figure \ref{Fig: no in prime thin} (i) and (ii) are subpolyominoes of $\MP$. Then we consider the inner intervals $[a,b]$ and $[b,e]$ of $\MP$  as in Figure \ref{fig:IJ}. \\
\begin{figure}[H]
\resizebox{0.45\textwidth}{!}{
\begin{tikzpicture}
\draw (0,0)--(2,0);
\draw (0,1)--(3,1);
\draw (1,2)--(3,2);

\draw (0,0)--(0,1);
\draw (1,0)--(1,2);
\draw (2,0)--(2,2);

\draw (3,1)--(3,2);

\node (d) at (1.5,-0.5)[below] {\Large{$\mathrm{(i)}$}};

\filldraw (0,0) circle (2pt) node [anchor=north]{$a$};
\filldraw (1,1) circle (2pt)node [anchor=south east]{$b$};
\filldraw (3,2) circle (2pt)node [anchor=west]{$e$};
\filldraw (3,1) circle (2pt)node [anchor=west]{$g$};
\filldraw (1,2) circle (2pt)node [anchor=east]{$f$};

\draw (0+5,0)--(1+5,0);
\draw (0+5,1)--(2+5,1);
\draw (0+5,2)--(2+5,2);
\draw (1+5,3)--(2+5,3);

\draw (0+5,0)--(0+5,2);
\draw (1+5,0)--(1+5,3);
\draw (2+5,1)--(2+5,3);

\node (d) at (1+5,-0.5)[below] {\Large{$\mathrm{(ii)}$}};

\filldraw (0+5,0) circle (2pt) node [anchor=north]{$a$};
\filldraw (1+5,1) circle (2pt)node [anchor=south east]{$b$};
\filldraw (1+5,3) circle (2pt)node [anchor=east]{$f$};
\filldraw (2+5,3) circle (2pt)node [anchor=west]{$e$};
\filldraw (2+5,1) circle (2pt)node [anchor=west]{$g$};

\end{tikzpicture}}\caption{}\label{fig:IJ}
\end{figure}
By Proposition \ref{prop:quadraticGB}, at least one between $[a,g]$ and $[a,f]$ is an inner interval of $\MP$, where $f$ and $g$ are the anti-diagonal corners of $[b,e]$. In both cases, we get a polyomino that is not thin, which is a contradiction.

(2) $\Rightarrow$ (1).
Assume, by contradiction, that $\MM$ does not form a quadratic Gr\"obner basis of $I_{\MP}$  with respect to $<_{\grevlex}^{i}$ for $i \in \mathcal{O}$. According to Proposition 2.1, there exist two inner intervals $[a,b]$ and $[b,e]$ of $\MP$, where $[a,b]$ has anti-diagonal corners $c$ and $d$, and $[b,e]$ has anti-diagonal corners $f$ and $g$, such that neither $[a,f]$ nor $[a,g]$ is an inner interval of $\MP$. 
Let $E$ and $F$ be respectively cells of $[a,b]$ and $[b,e]$ such that $E	\cap F =\{b\}$. Let $C$ and $D$ be respectively cells of $[a,f]$ and $[a,g]$ such that $E \cap C \cap D \cap F = \{b\}$. Since $\MP$ is thin, the cells $C$ and $D$ can not simultaneously be cells of $\MP$. If neither $C$ nor $D$ is a cell of $\PP$, then $C,D, E$, and $F$ are cells as in Figure \ref{fig:CD} (a) and this is a contradiction. Assume, without loss of generality, that $C \not \in \MP$, but $ D \in \MP$. Since $[a,g]$ is not an inner interval of $\MP$, then $d$ and $g$ are not both corners of $D$.\\
\begin{figure}[H]
\resizebox{0.5\textwidth}{!}{
\begin{tikzpicture}
\draw (0,0)--(2,0);
\draw (0,0)--(0,1);
\draw (0,1)--(2,1);

\draw (2,1)--(4,1);
\draw (2,2)--(4,2);
\draw (4,1)--(4,2);

\draw (2,0)--(2,2);

\draw (2,0)--(3,0) -- (3,1);

\node at (2.5,0.5){$D$};

\node (d) at (2,-0.5)[below] {\Large{$\mathrm{(i)}$}};

\filldraw (0,0) circle (2pt) node [anchor=north]{$a$};
\filldraw (2,1) circle (2pt)node [anchor=south east]{$b$};
\filldraw (4,2) circle (2pt)node [anchor=south]{$e$};
\filldraw (4,1) circle (2pt)node [anchor=north]{$g$};
\filldraw (2,0) circle (2pt)node [anchor=north]{$d$};
\filldraw (2,2) circle (2pt)node [anchor=south]{$f$};

\draw (6+2,0) -- (6+2,2) -- (7+2,2) -- (7+2,0) -- (6+2,0);
\draw (7+2,2) -- (7+2,4) -- (8+2,4) -- (8+2,2) -- (7+2,2);

\draw (9,1) -- (10,1) -- (10,2);
\node at (9.5,1.5){$D$};

\node (d) at (9,-0.5)[below] {\Large{$\mathrm{(ii)}$}};

\filldraw (8,0) circle (2pt) node [anchor=north]{$a$};
\filldraw (9,2) circle (2pt)node [anchor=south east]{$b$};
\filldraw (10,4) circle (2pt)node [anchor=west]{$e$};
\filldraw (10,2) circle (2pt)node [anchor=west]{$g$};
\filldraw (9,0) circle (2pt)node [anchor=north]{$d$};
\filldraw (9,4) circle (2pt)node [anchor=east]{$f$};

\end{tikzpicture}}\caption{}\label{Fig: union with D}
\end{figure}
Let $\MP'$ be the  subpolyomino of $\MP$ given by the union of the cells of $[a,b]$, $[b,e]$ and $D$, as in Figure \ref{Fig: union with D}. Then, one of the two subpolyominoes displayed in Figure \ref{Fig: no in prime thin} (i) and (ii) is a subpolyomino of $\MP'$, and then of $\MP$, which is a contradiction. 

\end{proof}

\begin{definition}\label{def:thincycle}
Let $\MP= \{C_1, \dots, C_n\}$ be a thin polyomino. If there exists a relabelling of the cells of $\PP$ such that $C_1,C_2,\ldots C_n$ is a path of cells, $C_1$ and $C_n$ have an edge in common, and $C_i \cap C_j = \emptyset$ for all $j>i+2$, then  $\MP$ is called \textit{thin cycle}.
\end{definition}

Note that a thin cycle is a polyomino with exactly one hole. In Figure \ref{Fig: thin cycles} three thin cycles are displayed. In particular, the polyominoes in (A) and (B) have the polyominoes in Figure \ref{Fig: no in prime thin} (i)--(iv) as subpolyominoes. This implies that in both cases $\MM$ is not a reduced Gr\"obner basis of $I_\MP$ with respect to $<_{\grevlex}^{i}$ for $i \in \{1,\ldots,8\}$. However, the polyomino in (A) is prime, whereas  the polyomino in (B) is not. Surprisingly, in the next result we exhibit a class of thin cycles having a prime ideal. The polyomino in Figure \ref{Fig: thin cycles} (C) belongs to such a class. 
\begin{figure}[H]
\centering
\begin{subfigure}{0.333 \textwidth}
\centering
\resizebox{0.6\textwidth}{!}{
\begin{tikzpicture}
\draw (0,0)--(0,3);
\draw (1,0)--(1,4);
\draw (2,0)--(2,1);
\draw (2,2)--(2,4);
\draw (3,0)--(3,1);
\draw (3,2)--(3,4);
\draw (4,0)--(4,4);
\draw (5,0)--(5,3);

\draw (0,0)--(5,0);
\draw (0,1)--(5,1);
\draw (0,2)--(2,2);
\draw (3,2)--(5,2);
\draw (0,3)--(5,3);
\draw (1,4)--(4,4);

\filldraw[gray, fill opacity=0.4] (0,0)--(1,0)--(1,3)--(0,3);
\filldraw[gray, fill opacity=0.4] (1,0)--(4,0)--(4,1)--(1,1);
\filldraw[gray, fill opacity=0.4] (4,0)--(5,0)--(5,3)--(4,3);
\filldraw[gray, fill opacity=0.4] (1,3)--(4,3)--(4,4)--(1,4);
\filldraw[gray, fill opacity=0.4] (1,2)--(2,2)--(2,3)--(1,3);
\filldraw[gray, fill opacity=0.4] (3,2)--(4,2)--(4,3)--(3,3);

\node (d) at (2.5,-1.5)[below] {\LARGE{$\mathrm{(A)}$}};

\end{tikzpicture}}
\end{subfigure}%
\begin{subfigure}{0.333 \textwidth}
\centering
\resizebox{0.6\textwidth}{!}{
\begin{tikzpicture}
\draw (0,0)--(0,3);
\draw (1,-1)--(1,4);
\draw (2,-1)--(2,1);
\draw (2,2)--(2,4);
\draw (3,-1)--(3,1);
\draw (3,2)--(3,4);
\draw (4,-1)--(4,4);
\draw (5,0)--(5,3);

\draw (1,-1)--(4,-1);
\draw (0,0)--(5,0);
\draw (0,1)--(2,1);
\draw (3,1)--(5,1);
\draw (0,2)--(2,2);
\draw (3,2)--(5,2);
\draw (0,3)--(2,3);
\draw (3,3)--(5,3);
\draw (1,4)--(4,4);

\filldraw[gray, fill opacity=0.4] (0,0)--(1,0)--(1,3)--(0,3);
\filldraw[gray, fill opacity=0.4] (1,-1)--(4,-1)--(4,0)--(1,0);
\filldraw[gray, fill opacity=0.4] (4,0)--(5,0)--(5,3)--(4,3);
\filldraw[gray, fill opacity=0.4] (1,3)--(4,3)--(4,4)--(1,4);
\filldraw[gray, fill opacity=0.4] (1,2)--(2,2)--(2,3)--(1,3);
\filldraw[gray, fill opacity=0.4] (3,2)--(4,2)--(4,3)--(3,3);
\filldraw[gray, fill opacity=0.4] (1,0)--(2,0)--(2,1)--(1,1);
\filldraw[gray, fill opacity=0.4] (3,0)--(4,0)--(4,1)--(3,1);

\node (d) at (2.5,-1.5)[below] {\LARGE{$\mathrm{(B)}$}};

\end{tikzpicture}}
\end{subfigure}%
\begin{subfigure}{0.333 \textwidth}
\centering
\resizebox{0.6\textwidth}{!}{
\begin{tikzpicture}
\draw (0,2)--(0,5);
\draw (1,2)--(1,5);
\draw (2,0)--(2,3);
\draw (2,4)--(2,7);
\draw (3,0)--(3,3);
\draw (3,4)--(3,7);
\draw (4,0)--(4,3);
\draw (4,4)--(4,7);
\draw (5,0)--(5,3);
\draw (5,4)--(5,7);
\draw (6,2)--(6,5);
\draw (7,2)--(7,5);

\draw (2,0)--(5,0);
\draw (2,1)--(5,1);
\draw (0,2)--(3,2);
\draw (4,2)--(7,2);
\draw (0,3)--(3,3);
\draw (4,3)--(7,3);
\draw (0,4)--(3,4);
\draw (4,4)--(7,4);
\draw (0,5)--(3,5);
\draw (4,5)--(7,5);
\draw (2,6)--(5,6);
\draw (2,7)--(5,7);

\filldraw[gray, fill opacity=0.4] (0,2)--(1,2)--(1,5)--(0,5);
\filldraw[gray, fill opacity=0.4] (6,2)--(7,2)--(7,5)--(6,5);
\filldraw[gray, fill opacity=0.4] (2,0)--(5,0)--(5,1)--(2,1);
\filldraw[gray, fill opacity=0.4] (2,6)--(5,6)--(5,7)--(2,7);
\filldraw[gray, fill opacity=0.4] (1,2)--(2,2)--(2,3)--(1,3);
\filldraw[gray, fill opacity=0.4] (1,4)--(2,4)--(2,5)--(1,5);
\filldraw[gray, fill opacity=0.4] (5,4)--(6,4)--(6,5)--(5,5);
\filldraw[gray, fill opacity=0.4] (5,2)--(6,2)--(6,3)--(5,3);
\filldraw[gray, fill opacity=0.4] (2,1)--(3,1)--(3,3)--(2,3);
\filldraw[gray, fill opacity=0.4] (4,1)--(5,1)--(5,3)--(4,3);
\filldraw[gray, fill opacity=0.4] (2,4)--(3,4)--(3,6)--(2,6);
\filldraw[gray, fill opacity=0.4] (4,4)--(5,4)--(5,6)--(4,6);

\node (d) at (3.5,-1)[below] {\Huge{$\mathrm{(C)}$}};

\end{tikzpicture}}
\end{subfigure}\caption{Examples of thin cycle polyominoes.}\label{Fig: thin cycles}
\end{figure}
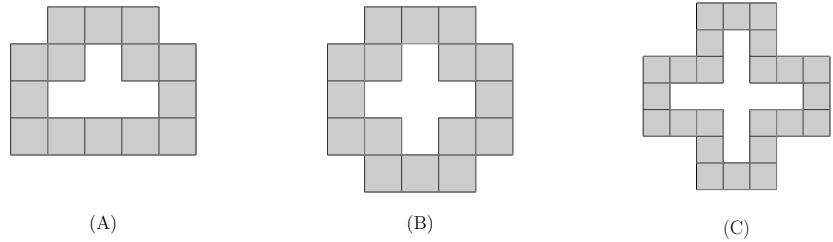

\begin{corollary}\label{cor: thin cycle prime}
Let $\MP$ be a thin cycle polyomino whose all maximal inner intervals have length at least $3$. Then $I_{\MP}$ is prime.
\end{corollary}
\begin{proof}
First of all, we observe that such a $\MP$ satisfies the condition (2) of Theorem \ref{theo: nec and suff for thin}. In fact, by definition of thin cycle, there are no cells $C,D,E$ and $F$ such that $E, F \in \MP$ intersect in one vertex, $C, D \not \in \MP$ and $C \cap D \cap E \cap F \not = \emptyset$, as in Figure \ref{fig:CD}. Moreover, by hypothesis, there is no maximal  inner intervals of length 2 as in Figure  \ref{Fig: no in prime thin}. By Theorem \ref{theo: nec and suff for thin}, $\MM$ is a quadratic Gr\"oebner basis for $I_{\MP}$ with respect to $<_{\grevlex}^i$, for all $i \in \{1, \dots, 8\}$. By Corollary \ref{Cor: quadratic GB implies prime}, the thesis follows. 
\end{proof}

As another application of the results obtained for thin polyominoes, we consider the grid polyominoes, that we introduced in \cite{MRR}. They are prime and, by definition, thin. One can see, by applying Proposition \ref{prop:quadraticGB}, that grid polyominoes have quadratic Gr\"obner basis with respect to $<_{\grevlex}^i$, for all $i \in \{1, \dots, 8\}$. In the following, we will define a new infinite family of prime polyominoes, obtained by the deletion of certain cells from grid polyominoes. We recall the following definition.
\begin{definition}\label{def: grid}
Let $\MP \subseteq I:=[(1,1),(m,n)]$ be a polyomino such that
\[
\MP=I \setminus \{\MH_{ij}: i \in [r], \ j \in [s] \},
\]
where $\MH_{ij}=[a_{ij},b_{ij}]$, with $a_{ij} =((a_{ij})_1,( a_{ij})_2)$, $b_{ij} =((b_{ij})_1,( b_{ij})_2)$, $1<(a_{ij})_1 < (b_{ij})_1 < m$,  $1<(a_{ij})_2 < (b_{ij})_2 < n$, and 
\begin{enumerate}
\item for any $i \in [r]$ and $\ell,k \in [s]$ we have $(a_{i\ell})_1=(a_{ik})_1$ and $(b_{i\ell})_1=(b_{ik})_1$;
\item for any $j \in [s]$ and $\ell,k \in [r]$ we have $(a_{\ell j})_2=(a_{kj})_2$ and $(b_{\ell j})_2=(b_{kj})_2$;
\item for any $i \in [r-1]$ and $\ j\in [s-1]$, we have $(a_{i+1j})_1 = (b_{ij})_1 +1 $ and $(a_{ij+1})_2 = (b_{ij})_2 +1$.
\end{enumerate}
We call $\MP$ a \emph{grid polyomino}.
\end{definition}

In Figure \ref{fig:grid}, an example of grid polyomino is displayed. For more examples, we refer readers to \cite{MRR}.

   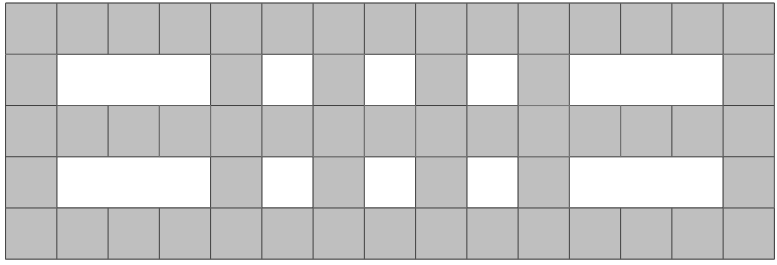
\begin{figure}[H]
   \resizebox{!}{0.27 \textwidth}{
   \begin{tikzpicture}
   \draw(0,0)--(15,0);
   \draw(0,5)--(15,5);
   \draw(0,1)--(15,1);
   \draw(0,2)--(15,2);
   \draw(0,3)--(15,3);
   \draw(0,4)--(15,4);
   
   \draw(0,0)--(0,5);
   \draw(1,0)--(1,5);
   \draw(2,0)--(2,1);
    \draw(2,2)--(2,3);
     \draw(2,4)--(2,5);
   \draw(3,0)--(3,1);
    \draw(3,2)--(3,3);
     \draw(3,4)--(3,5);
     \draw(4,0)--(4,5);
      \draw(5,0)--(5,5);
      \draw(6,0)--(6,5);
   \draw(7,0)--(7,5);
   \draw(8,0)--(8,5);
   \draw(9,0)--(9,5);
   \draw(10,0)--(10,5);
      \draw(11,0)--(11,5);
   \draw(12,0)--(12,1);
    \draw(12,2)--(12,3);
     \draw(12,4)--(12,5);
        \draw(13,0)--(13,1);
    \draw(13,2)--(13,3);
     \draw(13,4)--(13,5);
   \draw(14,0)--(14,5);
   \draw(15,5)--(15,0);
   
   \filldraw[opacity=0.5,gray] (0,0)--(15,0)--(15,1)--(0,1);
   \filldraw[opacity=0.5,gray] (0,4)--(15,4)--(15,5)--(0,5);
    \filldraw[opacity=0.5,gray] (0,1)--(1,1)--(1,4)--(0,4);
     \filldraw[opacity=0.5,gray] (14,1)--(15,1)--(15,4)--(14,4);
     \filldraw[opacity=0.5,gray] (4,1)--(5,1)--(5,4)--(4,4);
      \filldraw[opacity=0.5,gray] (1,2)--(3,2)--(3,3)--(1,3);
           \filldraw[opacity=0.5,gray] (5,2)--(6,2)--(6,3)--(5,3);
           \filldraw[opacity=0.5,gray] (6,2)--(7,2)--(7,4)--(6,4);
             \filldraw[opacity=0.5,gray] (8,2)--(9,2)--(9,4)--(8,4);
              \filldraw[opacity=0.5,gray] (9,2)--(11,2)--(11,3)--(9,3);
           \filldraw[opacity=0.5,gray] (10,1)--(11,1)--(11,2)--(10,2);
                   \filldraw[opacity=0.5,gray] (12,2)--(14,2)--(14,3)--(12,3);

                    \filldraw[opacity=0.5,gray] (3,2)--(4,2)--(4,3)--(3,3);
                      \filldraw[opacity=0.5,gray] (6,1)--(7,1)--(7,2)--(6,2);
                     \filldraw[opacity=0.5,gray] (7,2)--(8,2)--(8,3)--(7,3);
                       \filldraw[opacity=0.5,gray] (8,1)--(9,1)--(9,2)--(8,2);
                     \filldraw[opacity=0.5,gray] (10,3)--(11,3)--(11,4)--(10,4);
                      \filldraw[opacity=0.5,gray] (11,2)--(12,2)--(12,3)--(11,3);
   \end{tikzpicture}
   }
  \caption{A grid polyomino $\MP$. }\label{fig:grid}
  
  \end{figure}

\begin{definition}\label{def: subgrid}
A grid polyomino $\MP$ is the disjoint union of two collections of cells, namely $\MP=\mathcal{P}_1\sqcup \mathcal{P}_2$,
where $\mathcal{P}_1=\{C\in\MP \mid C$ is properly contained in exactly one maximal inner interval of $\MP\}$ and
$\mathcal{P}_2=\{C\in\MP \mid C$ is properly contained in $2$ maximal inner intervals of $\MP\}$. Let $\MP_1'$ be a subset of $\MP_1$ such that $\MP'=\MP\setminus \MP_1'$ is a polyomino. We call $\MP'$ a \textit{subgrid polyomino} of $\MP$.
\end{definition}

\begin{corollary}\label{cor: gridDeletion} Let $\MP'$ be a subgrid polyomino of a grid polyomino $\MP$. Then $I_{\MP'}$ is prime.
\end{corollary}
\begin{proof}
First of all, we claim that $\MP'$ satisfies the condition (2) of Theorem \ref{theo: nec and suff for thin}. By contradiction, assume that there exist $E$ and $F$ cells of $\MP'$ as in Figure \ref{fig:CD}, but neither $C$ nor $D$ is a cell of $\MP'$. By definition of grid polyomino, either $C$ or $D$ is a cell of $\MP$. Without loss of generality, we may assume that $C$ is a cell of $\MP$. Then $C \in \MP_2$, and $C \not \in \MP_1$, since $C$ is properly contained in two maximal inner intervals: one containing the cells $C$ and $E$ and the other containing $C$ and $F$. Then $C$ is still a cell of $\MP'$. Moreover, by definition of grid polyomino, the subpolyominoes displayed in Figure \ref{Fig: no in prime thin} are not subpolyominoes of $\MP$. Since $\MP' \subset \MP$, then they are not subpolyominoes of $\MP'$ either. By Theorem \ref{theo: nec and suff for thin}, $\MM$ is a quadratic Gr\"oebner basis for $I_{\MP'}$ with respect to $<_{\grevlex}^i$, for all $i \in \{1, \dots, 8\}$. By Corollary \ref{Cor: quadratic GB implies prime}, the thesis follows. 
\end{proof}

In Figure \ref{fig:gridDeletion}, it is shown a subgrid polyomino $\MP'$ obtained from the grid polyomino $\MP$ displayed in Figure \ref{fig:grid} by removing some cells in $\MP_1$. By Corollary \ref{cor: gridDeletion}, the ideal $I_{\MP'}$ is prime. 

   \begin{figure}[H]
      \resizebox{!}{0.27 \textwidth}{
   \begin{tikzpicture}
   \draw(0,0)--(1,0);
   \draw(2,0) -- (15,0);
   \draw(0,5)--(13,5);
   \draw (14,5) -- (15,5);
   \draw(0,1)-- (1,1);
   \draw (2,1) -- (15,1);
   \draw(0,2)--(3,2);
   \draw (4,2) -- (7,2);
         \draw (8,2) -- (11,2);
   \draw (12,2) -- (15,2);
   \draw(0,3)--(3,3);
   \draw (4,3) -- (7,3);
      \draw (8,3) -- (11,3);
   \draw (12,3) -- (15,3);
   \draw(0,4)-- (13,4);
   \draw (14,4) -- (15,4);
  
   \draw(0,0)--(0,5);
   \draw(1,0)-- (1,5);
   \draw(2,0)--(2,1);
   \draw (2,2) --  (2,3);
   \draw (2,4) -- (2,5);
     \draw(3,0)--(3,1);
   \draw (3,2) --  (3,3);
   \draw (3,4) -- (3,5);
     \draw(4,0)--(4,5);
      \draw(5,0)--(5,5);
      \draw(6,0)--(6,1);
         \draw(6,2)--(6,5);
      \draw(7,0)--(7,1);
         \draw(7,2)--(7,5);
               \draw(8,0)--(8,1);
         \draw(8,2)--(8,5);
   \draw(9,0)--(9,1);
      \draw(9,2)--(9,5);

      \draw(10,0)--(10,1);
   \draw(10,1)--(10,5);
   \draw(11,0)--(11,5);
 \draw(12,0)--(12,1);
   \draw (12,2) --  (12,3);
   \draw (12,4) -- (12,5);
\draw(13,0)--(13,1);
   \draw (13,2) --  (13,3);
   \draw (13,4) -- (13,5);
      \draw(14,0)--(14,5);
   \draw(15,5)--(15,0);
   
      \filldraw[opacity=0.5,gray] (0,0)--(1,0)--(1,1)--(0,1);
   \filldraw[opacity=0.5,gray] (2,0)--(15,0)--(15,1)--(2,1);
   \filldraw[opacity=0.5,gray] (0,4)--(13,4)--(13,5)--(0,5);
   \filldraw[opacity=0.5,gray] (14,4)--(15,4)--(15,5)--(14,5);
    \filldraw[opacity=0.5,gray] (0,1)--(1,1)--(1,4)--(0,4);
     \filldraw[opacity=0.5,gray] (14,1)--(15,1)--(15,4)--(14,4);
     \filldraw[opacity=0.5,gray] (4,1)--(5,1)--(5,4)--(4,4);
      \filldraw[opacity=0.5,gray] (1,2)--(3,2)--(3,3)--(1,3);
            \filldraw[opacity=0.5,gray] (10,3)--(11,3)--(11,4)--(10,4);
           \filldraw[opacity=0.5,gray] (5,2)--(6,2)--(6,3)--(5,3);
           \filldraw[opacity=0.5,gray] (6,2)--(7,2)--(7,4)--(6,4);
             \filldraw[opacity=0.5,gray] (8,2)--(9,2)--(9,4)--(8,4);
              \filldraw[opacity=0.5,gray] (9,2)--(11,2)--(11,3)--(9,3);
                \filldraw[opacity=0.5,gray] (10,1)--(11,1)--(11,2)--(10,2);
                   \filldraw[opacity=0.5,gray] (12,2)--(14,2)--(14,3)--(12,3);
   \end{tikzpicture}
}
 \caption{A subgrid polyomino of the grid polyomino in Figure \ref{fig:grid}.}\label{fig:gridDeletion}

\end{figure}
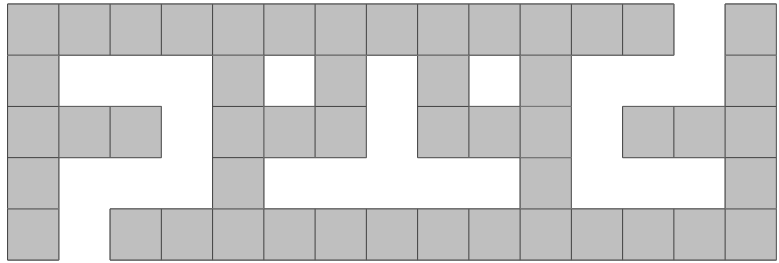

\end{document}